\documentclass[11pt]{article}
\usepackage{amssymb,amsthm,amsmath,float,cite,enumitem}
\usepackage[linesnumbered,ruled,lined]{algorithm2e}
\usepackage[margin=1in]{geometry}
\usepackage{tikz}
\usepackage[colorlinks=true,citecolor=black,linkcolor=black,urlcolor=blue]{hyperref}

\tikzstyle{hvertex}=[thick,circle,inner sep=0.cm, minimum size=2mm, fill=white, draw=black]
\tikzstyle{hedge}=[ultra thin,draw opacity=.4]
\colorlet{hellgrau}{black!30!white}

\newtheorem{theorem}{Theorem}
\newtheorem{lemma}[theorem]{Lemma}

\newtheorem{proposition}[theorem]{Proposition}

\newtheorem{claim}{Claim}
\newtheorem*{claim*}{Claim}
\newtheorem{case}{Case}

\usepackage{setspace}
\usepackage{lineno}

\begin{document}

\onehalfspace

\title{Rainbow triangles and cliques in edge-colored graphs}
 
\author{ Stefan Ehard \and Elena Mohr}

\date{}

\maketitle

\begin{center}
Institut f\"{u}r Optimierung und Operations Research, 
Universit\"{a}t Ulm, Ulm, Germany,
\{\texttt{stefan.ehard,elena.mohr}\}\texttt{@uni-ulm.de}\\[3mm]
\end{center}

\begin{abstract}
For an edge-colored graph, a subgraph is called {\it rainbow} if all its edges have distinct colors.
We show that if $G$ is an edge-colored graph of order $n$ and size $m$ using $c$ colors on its edges, and $m+c\geq \binom{n+1}{2}+k-1$ for a non-negative integer $k$, then $G$ contains at least $k$ rainbow triangles.
For $n\geq 3k$, we show that this result is best possible, and we completely characterize the class of edge-colored graphs for which this result is sharp. Furthermore, we show that an edge-colored graph $G$ contains at least $k$ rainbow triangles if $\sum\limits_{v\in V(G)} d^c_G(v)\geq \binom{n+1}{2}+k-1$ where $d_G^c(v)$ denotes the number of distinct colors incident to a vertex $v$.

Finally we characterize the edge-colored graphs without a rainbow clique of size at least six that maximize the sum of edges and colors $m+c$.

Our results answer two questions of Fujita, Ning, Xu and Zhang [On sufficient conditions for rainbow cycles in edge-colored graph, arXiv:1705.03675, 2017]

\end{abstract}

{\small 
\begin{tabular}{lp{13cm}}
 \end{tabular}
}

\section{Introduction}
In 1907, Mantel \cite{ma} proved that the maximum number of edges in a triangle-free graph on $n$ vertices is at most $\lfloor\frac{n^2}{4}\rfloor$, which was generalized by the famous theorem of Tur\'an \cite{tu} in 1941.
Rademacher showed in 1941 that a graph on $n$ vertices and at least $\lfloor\frac{n^2}{4}\rfloor+1$ edges contains not only one but already at least $\lfloor\frac{n}{2}\rfloor$ triangles.
In 1955, Erd\H{o}s\cite{er} extended Rademacher's result by showing that a graph on $n$ vertices and at least $\lfloor\frac{n^2}{4}\rfloor+k$ edges contains at least $k\lfloor\frac{n}{2}\rfloor$ triangles for $k<\min\{4,n/2\}$. 
Erd\H{o}s also conjectured the same to be true for $k< n/2$, which was later proved by Lov\'asz and Simonovits~\cite{losi}.
In the present paper, we investigate {\it rainbow} versions of these extremal problems.
A subgraph of an edge-colored graph is called {\it rainbow} if all its edges have different colors.
Erd\H{o}s, Simonovits and S\'os~\cite{ersiso} proved that a complete graph on $n$ vertices that is edge-colored with at least $n$ colors contains a rainbow triangle, which is best possible.
Considering the number of edges $m(G)$ of a graph $G$ and the number of colors $c(G)$ that are used to color the edges of $G$, Li, Ning, Xu, and Zhang~\cite{linixuzh} extended this and showed the following rainbow version of Mantel's theorem.
\begin{theorem}[Li et al.~\cite{linixuzh}]\label{thm:li}
If $G$ is an edge-colored graph on $n$ vertices such that $m(G)+c(G)\geq \binom{n+1}{2}$, then $G$ contains a rainbow triangle.
\end{theorem}

Observe that there exist edge-colored (complete) graphs with $m(G)+c(G)=\binom{n+1}{2}-1$ that do not contain a rainbow triangle.
Fujita, Ning, Xu, and Zhang~\cite{funixuzh} characterized all extremal graphs $G$ for Theorem~\ref{thm:li} such that $m(G)+c(G)$ is maximum but $G$ contains at most one rainbow triangle.
In contrast to Rademacher's theorem, which states that a graph with  $\lfloor\frac{n^2}{4}\rfloor+1$ edges contains at least $\lfloor\frac{n^2}{4}\rfloor$ triangles, a graph $G$ with $m(G)+c(G)=\binom{n+1}{2}$ may contain only one rainbow triangle.
This motivates the question in~\cite{funixuzh} whether one can find a function $f(k)$ such that a graph $G$ on $n$ vertices with $m(G)+c(G)\geq \binom{n+1}{2}+f(k)$ necessarily contains $k$ rainbow triangles, which can be seen as a rainbow version of Erd\H{o}s' conjecture mentioned above. 
Our first result answers this question.

\begin{theorem}\label{thm:m+c}
If $G$ is an edge-colored graph on $n$ vertices such that $m(G)+c(G)\geq \binom{n+1}{2}+k-1$ for a positive integer $k$, then $G$ contains $k$ rainbow triangles.
\end{theorem}
For the complete graph $K_n$, this extends the above mentioned result of Erd\H{o}s, Simonovits and S\'os~\cite{ersiso} and implies the existence of $k$ rainbow triangles in $K_n$ if $c(K_n)\geq n+k-1$. 

To see that Theorem~\ref{thm:m+c} is best possible for $n\geq 3k$, consider the following edge-colored complete graph on $n\geq 3k$ vertices using $n+k-1$ distinct colors that contains only $k$ rainbow triangles:
\begin{itemize}
\item Let $G_0$ be the complete graph with vertex set $\{v_1,\dots,v_{n-3k}\}$, where the edge $v_iv_j$ is colored with color $i$ for $i<j$. 
\item For $i$ form 1 up to $k$ let $G_i$ arise from $G_{i-1}$ by adding three vertices that form a rainbow triangle, whose colors are not yet used in $G_{i-1}$. We add all edges between $G_{i-1}$ and the triangle and color them with another new color which is neither used in $G_{i-1}$ nor in the new triangle. 
\end{itemize}
It is easy to see that $G_k$ contains exactly $k$ rainbow triangles, and that $m(G)+c(G)=\binom{n}{2}+n-3k-1+4k=\binom{n+1}{2}+k-1$.

\begin{figure}[H]

\begin{center}
\begin{tikzpicture}[scale=1.5]

\def\krad{0.5}

\def\angle{360/3}
\begin{scope}[shift={(-0.5,-1.5)}];
\node[hvertex] (w1) at (270-\angle/2+0*\angle:\krad){};
\node[hvertex] (w2) at (270-\angle/2+1*\angle:\krad){};
\node[hvertex] (w3) at (270-\angle/2+2*\angle:\krad){};
\end{scope}

\begin{scope}[shift={(3.5,-1.5)}];
\node[hvertex] (x1) at (270-\angle/2+0*\angle:\krad){};
\node[hvertex] (x2) at (270-\angle/2+1*\angle:\krad){};
\node[hvertex] (x3) at (270-\angle/2+2*\angle:\krad){};
\end{scope}

\node[hvertex] (y1) at (0,0){};
\node[hvertex] (y2) at (1,0){};
\node[hvertex] (y3) at (2,0){};
\node[hvertex] (y4) at (3,0){};

\node (h1) at (0,-0.2){};
\node (h2) at (0.3,-0.2){};
\node (h3) at (0.6,-0.2){};
\node (h4) at (0.9,-0.2){};

\node (g1) at (-0.3,-1.2){};
\node (g2) at (0,-1.45){};
\node (g3) at (0.2,-1.7){};

\foreach \a in {1,...,4}
\foreach \x in {1,...,3}
    \draw[thick, color=orange] (h\a)--(g\x);

\node (f1) at (0.5,-1.1){};
\node (f2) at (0.5,-1.4){};
\node (f3) at (0.5,-1.7){};

\node (e1) at (2.5,-1.1){};
\node (e2) at (2.5,-1.4){};
\node (e3) at (2.5,-1.7){};

\foreach \a in {1,...,3}
\foreach \x in {1,...,3}
    \draw[thick, color=yellow] (e\a)--(f\x);

\node (i1) at (2,-0.2){};
\node (i2) at (2.3,-0.2){};
\node (i3) at (2.6,-0.2){};
\node (i4) at (2.9,-0.2){};

\node (j1) at (3.3,-1.2){};
\node (j2) at (3,-1.45){};
\node (j3) at (2.8,-1.7){};

\foreach \a in {1,...,4}
\foreach \x in {1,...,3}
    \draw[thick, color=yellow] (i\a)--(j\x);

\draw[ultra thick, color=red] (w1) -- (w2);
\draw[ultra thick, color=red!50!black] (w1) -- (w3);
\draw[ultra thick, color=red!20!white] (w2) -- (w3);

\draw[ultra thick, color=green] (x1) -- (x2);
\draw[ultra thick, color=green!50!black] (x1) -- (x3);
\draw[ultra thick, color=green!30!white](x2) -- (x3);
  
\draw[ultra thick, color=black!70!blue] (y1) -- (y2);
\draw[ultra thick,  bend left, color=black!70!blue] (y1) to (y3);
\draw[ultra thick, bend left,color=black!70!blue] (y1) to (y4);
\draw[ultra thick, color=blue] (y2) -- (y3);
\draw[ultra thick, color=blue, bend left] (y2) to (y4);
\draw[ultra thick, color=white!60!blue] (y3) -- (y4);

\end{tikzpicture}
\end{center}
\caption{The graph $G_2$ on $10$ vertices.}\label{fig1}
\end{figure}
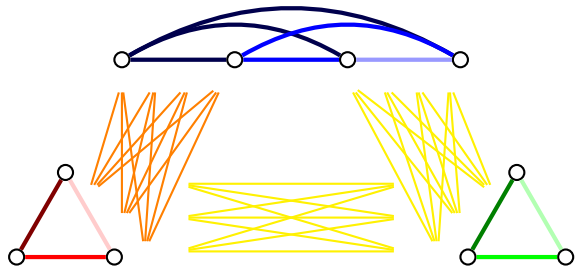

We can also characterize the graphs $G$ with $m(G)+c(G)\geq\binom{n+1}{2}+k-1$ that contain exactly $k$ rainbow triangles. This generalizes the main results of~\cite{funixuzh}, where this is proved for $k\in\{0,1\}$. 
To this end, we define the following class of graphs: 
Let $\mathcal{G}_k$ be the set of all edge-colored complete graphs that contain exactly $k$ rainbow triangles such that for every graph $G\in\mathcal{G}_k$ on $n$ vertices, $c(G)=n+k-1$, 
and either $G$ is a single vertex or a rainbow triangle, or there is a partition $(V_1,V_2)$ of the vertex set of $G$ such that the edges between $V_1$ and $V_2$ all have the same color, $G[V_1]\in\mathcal{G}_i$, and $G[V_2]\in\mathcal{G}_j$ for some non negative integers $i$ and $j$ with $i+j=k$.
Note that this definition implies that all rainbow triangles of a graph in $\mathcal{G}_k$ are vertex disjoint, and that the graph $G_k$ constructed above belongs to $\mathcal{G}_k$.

\begin{theorem}\label{thm:char}
If $G$ is an edge-colored graph on $n$ vertices with $n\geq3k$ for a non-negative integer $k$, $m(G)+c(G)\geq\binom{n+1}{2}+k-1$ and $G$ contains exactly $k$ rainbow triangles, then $G\in\mathcal{G}_k$.
\end{theorem}

Instead of forcing rainbow triangles by a global condition, one can also do this by a local condition.
Li et al.~\cite{linixuzh} proved that for an edge-colored graph $G$ on $n$ vertices, $\sum\limits_{v\in V(G)} d^c_G(v)\geq \binom{n+1}{2}$ implies the existence of a rainbow triangle where $d^c_G(v)$ denotes the {\it color degree}, the number of different colors among the edges incident to a vertex $v$ in $G$.
Note that there are edge-colored graphs on $n$ vertices that fulfill the sum condition, $\sum\limits_{v\in V(G)} d^c_G(v)\geq \binom{n+1}{2}$, but contain exactly one rainbow triangle. One such graph arises from the graph $G_1$ as introduced above, by recoloring the edges between every $v_i\in\{v_1,\ldots,v_{n-4}\}$ and the rainbow triangle with color $i$.

We extend the result of~\cite{linixuzh} and provide a lower bound of the same type which guarantees the existence of $k$ rainbow triangles.

\begin{theorem}\label{thm:sum}
If $G$ is an edge-colored graph on $n$ vertices such that 
$\sum\limits_{v\in V(G)} d^c_G(v)\geq \binom{n+1}{2}+k-1$ 
for a positive integer $k$, then $G$ contains $k$ rainbow triangles.
\end{theorem}

Similarly as Tur\'an generalized Mantel's theorem, it is natural to ask whether one can generalize the previous results concerning rainbow triangles to larger rainbow cliques.
In fact, Xu, Hu, Wang, and Zhang~\cite{xuhuwazh} recently extended a result of Montellano-Ballesteros and Neumann-Lara\cite{mone} as well as Schiermeyer~\cite{sc} in the following way. 
We write $t_{n,k}$ for the number of edges in the Tur\'an graph $T_{n,k}$, the complete $k$-partite graph on $n$ vertices whose sizes of the partite sets are as equal as possible.
\begin{theorem}[Xu et al.~\cite{xuhuwazh}]\label{thm:xucliques}
If $G$ is an edge-colored graph on $n$ vertices and $n\geq k\geq 4$ such that $m(G)+c(G)\geq \binom{n}{2}+t_{n,k-2}+2$, then $G$ contains a rainbow $K_k$.
\end{theorem}

As noticed in~\cite{sc, xuhuwazh}, this bound is sharp because of the following example. Consider a complete graph $H_{n,k-2}$ on $n$ vertices. 
Let $T_{n,k-2}$ be a subgraph of $H_{n,k-2}$ where all edges are colored differently, and all remaining edges in $H_{n,k-2}$ are colored alike with one further color. 

Answering another question of \cite{funixuzh}, we can in fact show that these are essentially all the extremal graphs if $k\geq6$, except when one partite set of $T_{n,k-2}$ has size 1.
To this end, let $\mathcal{H}_k$ be the set of all edge-colored graphs such that for every graph $H\in\mathcal{H}_k$ on $n$ vertices either
\begin{enumerate}[label=(\Roman*)]
\item\label{i} $H$ is isomorphic to $H_{n,k-2}$ or
\item\label{ii} $\lfloor n/(k-2)\rfloor=1$, $H$ is complete, $c(H)=t_{n,k-2}+1$, and $H$ contains a rainbow $T_{n,k-2}$ as a subgraph but no rainbow $K_k$.  
\end{enumerate} 

\begin{figure}[H]
\begin{center}
\begin{tikzpicture}

{\begin{scope}[]
\tikzstyle{enddot}=[circle,inner sep=0cm, minimum size=10pt,color=hellgrau,fill=hellgrau]
\tikzstyle{markline}=[draw=hellgrau,line width=10pt]

\def\csizes{{3,2,2,2,2}} 
\def\cnum{4}         
\def\radius{2}

\def\startangle{90}

\pgfmathsetmacro{\cnumminusone}{\cnum-1}
\pgfmathsetmacro{\classrange}{360/(\cnum+1)}


\foreach \i in {0,...,\cnum}{
  \pgfmathparse{\csizes[\i]}
  \pgfmathsetmacro{\gap}{\classrange/(\pgfmathresult+1)}
  
  \pgfmathsetmacro{\markangle}{0.8*\classrange}  

  \pgfmathsetmacro{\sangle}{\startangle+\i*\classrange+\classrange+0.1*\classrange}  
  \draw[markline] (\sangle:\radius) arc [radius=\radius, start angle=\sangle, delta angle=\markangle];

  \foreach \vx in {1,...,\pgfmathresult}{
    \node[hvertex] (a\i b\vx) at (\startangle+\i*\classrange+\classrange+\vx*\gap:\radius){};
  }
}

\foreach \i in {0,...,\cnumminusone}{
  \pgfmathtruncatemacro{\plusone}{\i+1}
  \foreach \j in {\plusone,...,\cnum}{
     \pgfmathparse{\csizes[\i]}
     \foreach \u in {1,...,\pgfmathresult}{
       \pgfmathparse{\csizes[\j]}
        \foreach \v in {1,...,\pgfmathresult}{
           \draw[hedge] (a\i b\u) -- (a\j b\v);
        }
     }
  }
}
\draw[ultra thick,red] (a0b1) -- (a0b2);
\draw[ultra thick,red] (a0b2) -- (a0b3);
\draw[ultra thick,bend right=40,red] (a0b1) to (a0b3);

\draw[ultra thick,red] (a1b1) -- (a1b2);

\draw[ultra thick,red] (a2b1) -- (a2b2);

\draw[ultra thick,red] (a3b1) -- (a3b2);

\draw[ultra thick,red] (a4b1) -- (a4b2);
\end{scope}}

{\begin{scope}[shift={(5cm,0cm)}]
\tikzstyle{enddot}=[circle,inner sep=0cm, minimum size=10pt,color=hellgrau,fill=hellgrau]
\tikzstyle{markline}=[draw=hellgrau,line width=10pt]

\def\csizes{{2,2,2,1,1,1}} 
\def\cnum{4}         
\def\radius{2}

\def\startangle{90}

\pgfmathsetmacro{\cnumminusone}{\cnum-1}
\pgfmathsetmacro{\classrange}{360/(\cnum+1)}


\foreach \i in {0,...,\cnum}{
  \pgfmathparse{\csizes[\i]}
  \pgfmathsetmacro{\gap}{\classrange/(\pgfmathresult+1)}
  
  \pgfmathsetmacro{\markangle}{0.8*\classrange}  

  \pgfmathsetmacro{\sangle}{\startangle+\i*\classrange+\classrange+0.1*\classrange}  
  \draw[markline] (\sangle:\radius) arc [radius=\radius, start angle=\sangle, delta angle=\markangle];

  \foreach \vx in {1,...,\pgfmathresult}{
    \node[hvertex] (a\i b\vx) at (\startangle+\i*\classrange+\classrange+\vx*\gap:\radius){};
  }
}

\foreach \i in {0,...,\cnumminusone}{
  \pgfmathtruncatemacro{\plusone}{\i+1}
  \foreach \j in {\plusone,...,\cnum}{
     \pgfmathparse{\csizes[\i]}
     \foreach \u in {1,...,\pgfmathresult}{
       \pgfmathparse{\csizes[\j]}
        \foreach \v in {1,...,\pgfmathresult}{
           \draw[hedge] (a\i b\u) -- (a\j b\v);
        }
     }
  }
}

\draw[ultra thick,blue] (a2b1) -- (a2b2);
\draw[ultra thick,blue] (a2b2) -- (a3b1);

\draw[ultra thick,color=green!50!black] (a1b1) -- (a1b2);
\draw[ultra thick,color=green!50!black] (a1b1) -- (a4b1);

\draw[ultra thick,red] (a0b1) -- (a0b2);

\end{scope}}

\end{tikzpicture}
\end{center}
\caption{\onehalfspacing Two edge-colored complete graphs without a rainbow $K_7$ such that $m(G)+c(G)=\binom{n}{2}+t_{n,5}+1$. 
The thin lines represent the edges of a rainbow $T_{n,5}$ and use different colors than the colors shown in the figure. 
The left graph $H_{11,5}$ corresponds to case \ref{i}, and the right one to case \ref{ii} for $n=8$ where two edges within partite sets of size 2 are colored with a color of an edge incident to a partite set of size 1.}\label{fig2}
\end{figure}
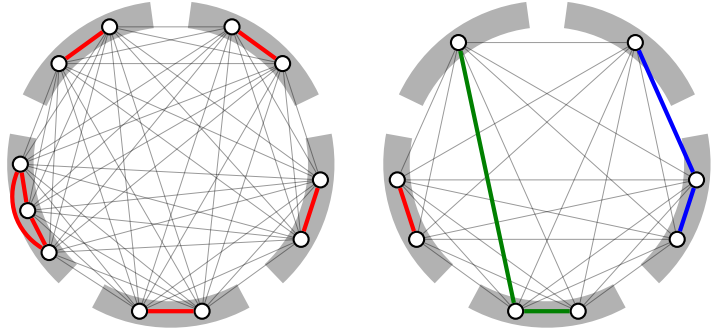

Note that all graphs in $\mathcal{H}_k$ contain a rainbow $T_{n,k-2}$ and are complete graphs.
Case \ref{ii} attributes the case, where partite sets of size 1 allow to suitably color edges within the partite sets of size 2 with colors that are used in the rainbow $T_{n,k-2}$.

Our third main result is as follows.

\begin{theorem}\label{thm:Kkchar}
If $G$ is an edge-colored graph on $n$ vertices and $n\geq k\geq 6$ such that $m(G)+c(G)=\binom{n}{2}+t_{n,k-2}+1$ and $G$ does not contain a rainbow $K_k$, then $G\in\mathcal{H}_k$. 
\end{theorem}

We introduce some additional notation.
Throughout this paper we consider finite, simple, (not necessarily properly) edge-colored graphs, usually of order $n$, and use standard terminology. 
Let $d^s_G(v)$ denote the {\it saturated degree} of a graph $G$, $d^s_G(v)=c(G)-c(G-v)$, that is the number of unique colors incident to a vertex $v$. A {\it v-saturated edge} $e$ is an edge incident with $v$ such that the color of $e$ is not used in $G-v$.
We will frequently use the fact that $\sum\limits_{v\in V(G)}d^s_G(v) \leq 2c(G)$.

The proofs of Theorems~\ref{thm:m+c}, \ref{thm:char} and \ref{thm:sum} are given in Section~\ref{sec:2}. 
Theorem~\ref{thm:Kkchar} is proved in Section~\ref{sec:3}.

\section{Rainbow triangles in edge-colored graphs}\label{sec:2}
We begin with the proof of Theorem~\ref{thm:m+c}.
\begin{proof}[Proof of Theorem~\ref{thm:m+c}]
Suppose, for a contradiction, that the statement is false, and let $G$ be a counterexample chosen such that 
\begin{enumerate}[label=(\roman*)]
\item $k$ is as small as possible,
\item subject to (i), the number of vertices $n$ is as small as possible, and 
\item subject to (i) and (ii), the number of edges $m(G)$ is as small as possible.
\end{enumerate}

By Theorem~\ref{thm:li}, we know that $k\geq2$, and thus by the choice of $G$, we can find at least $k-1$ rainbow triangles in $G$.
\begin{claim}\label{claim1}
$m(G)+c(G)=\binom{n+1}{2}+k-1$.
\end{claim}
\begin{proof}
Suppose, for a contradiction, that $m(G)+c(G)\geq\binom{n+1}{2}+k$ and let $e$ be an edge of a rainbow triangle in $G$. Since $m(G-e)+c(G-e)\geq\binom{n+1}{2}+k-2$, $G-e$ contains at least $k-1$ rainbow triangles by the choice of $G$. 
Hence, $G$ contains at least $k$ rainbow triangles, a contradiction.
\end{proof}

\begin{claim}\label{claim2}
For every vertex $v$ of $G$ that is contained in $\ell\geq0$ rainbow triangles, $d_G(v)+d_G^s(v)\geq n+\ell+1.$
\end{claim}
\begin{proof}
Suppose, for a contradiction, that there is a vertex $v$ that is contained in $\ell$ rainbow triangles and $d_G(v)+d_G^s(v)\leq n+\ell$.
Since $m(G-v)+c(G-v)=m(G)+c(G)-d_G(v)-d_G^s(v)
\geq 
\binom{n}{2}+k-\ell-1$, $G-v$ contains at least $k-\ell$ rainbow triangles by the choice of $G$.
Hence, $G$ contains at least $k$ rainbow triangles, a contradiction.
\end{proof}

Since $G$ contains $k-1$ rainbow triangles, Claims~\ref{claim1} and~\ref{claim2} imply that
$$n(n+1)+3(k-1)\leq \sum\limits_{v\in V(G)} \Big( d_G(v)+d_G^s(v) \Big) \leq 2m(G)+2c(G)=n(n+1)+2(k-1),$$
which yields the final contradiction as $k\geq2$.
\end{proof}

\setcounter{claim}{0}
In order to prove Theorem~\ref{thm:char}, we need the following lemma.
\begin{lemma}\label{lemma}
If $G$ is an edge-colored graph on $n$ vertices, $m(G)+c(G)\geq\binom{n+1}{2}+k-1$ and $G$ contains exactly $k$ rainbow triangles, then $m(G)+c(G)=\binom{n+1}{2}+k-1$ and $G$ is complete.
\end{lemma}
\begin{proof}
Suppose, for a contradiction, that the statement is false, and let $G$ be a counterexample chosen such that  
the number of vertices $n$ is as small as possible.
Since $G$ contains exactly $k$ rainbow triangles, Theorem~\ref{thm:m+c} implies that $m(G)+c(G)=\binom{n+1}{2}+k-1$, so we may assume that $G$ is not complete.
\begin{claim}\label{claimLemma}
For every vertex $v$ of $G$ that is contained in $\ell\geq0$ rainbow triangles, it holds that $d_G(v)+d_G^s(v)\geq n+\ell+1.$
\end{claim}
\begin{proof}[Proof of Claim~\ref{claimLemma}]
Suppose, for a contradiction, that there is a vertex $v$ that is contained in exactly $\ell$~rainbow triangles and $d_G(v)+d_G^s(v)\leq n+\ell$.
If $d_G(v)+d_G^s(v)\leq n+\ell-1$, then $m(G-v)+c(G-v)\geq \binom{n}{2}+k-\ell$, and thus, by Theorem~\ref{thm:m+c}, $G-v$ contains at least $k-\ell+1$ rainbow triangles, which implies that $G$ contains at least $k+1$ rainbow triangles, a contradiction.
So, we may assume that $d_G(v)+d_G^s(v)=n+\ell$. 
This implies that $m(G-v)+c(G-v)\geq \binom{n}{2}+k-\ell-1$.
Since $G-v$ contains exactly $k-\ell$ rainbow triangles, the choice of $G$ implies that $G-v$ must be a complete graph. Since $G$ is not complete, $d_G(v)\leq n-2$.
Hence, $d_G^s(v)\geq \ell+2$ and thus $v$ is contained in at least $\binom{\ell+2}{2}>\ell$ rainbow triangles in $G$, a contradiction, which proves the claim.
\end{proof}
Now, similarly as in the proof of Theorem~\ref{thm:m+c}, we obtain that 
$$n(n+1)+3k\leq \sum\limits_{v\in V(G)}\Big( d_G(v)+d_G^s(v) \Big)\leq 2m(G)+2c(G)=n(n+1)+2(k-1),$$
which yields the final contradiction as $k\geq1$.
\end{proof}
\setcounter{claim}{0}
\begin{proof}[Proof of Theorem~\ref{thm:char}]
Suppose, for a contradiction, that the statement is false, and let $G$ be a counterexample chosen such that 
\begin{enumerate}[label=(\roman*)]
\item $k$ is as small as possible,
\item subject to (i), the number of vertices $n$ is as small as possible, and 
\item subject to (i) and (ii), the number of edges $m(G)$ is as small as possible.
\end{enumerate}
Since the case $k=0$ was already shown in~\cite{funixuzh}, we may assume that $k\geq1$ and that $n\geq4$.
By Lemma~\ref{lemma}, $G$ is complete and $m(G)+c(G)=\binom{n+1}{2}+k-1$.
\begin{claim}\label{claimEdge}
Each color of a rainbow triangle appears on only one edge in $G$.
\end{claim}
\begin{proof}[Proof of Claim~\ref{claimEdge}]
Suppose, for a contradiction, that there is an edge $e$ with $c(G)=c(G-e)$ that is contained in a rainbow triangle of $G$.
Hence, $m(G-e)+c(G-e)=\binom{n+1}{2}+k-2$, and thus, by Theorem~\ref{thm:m+c} and since $e$ is contained in one of the $k$ rainbow triangles of $G$, $G-e$ contains exactly $k-1$ rainbow triangles. 
The choice of $G$ implies that $G-e\in\mathcal{G}_{k-1}$ and thus $G-e$ must be complete, a contradiction, which proves the claim.
\end{proof}
Now, let $v_1v_2v_3$ be a rainbow triangle in $G$.
Let $H$ arise from $G$ by recoloring the edge $v_1v_2$ with the color of the edge $v_2v_3$. 
By Claim~\ref{claimEdge}, this operation does not create a new rainbow triangle and deletes exactly one rainbow triangle of $G$.
Therefore, $H$ contains exactly $k-1$ rainbow triangles and $m(H)+c(H)=\binom{n+1}{2}+k-2$.
By the choice of $G$, we obtain $H\in\mathcal{G}_{k-1}$.
Thus, there is a non-trivial partition $(V_1,V_2)$ of the vertex set of $H$ such that the edges between $V_1$ and $V_2$ are all colored the same in $H$, and $H[V_1]\in\mathcal{G}_i$, $H[V_2]\in\mathcal{G}_j$ for some $i,j\geq 0$ with $i+j=k-1$.
Since $n\geq4$, we may assume that $\{v_1,v_2,v_3\}\subseteq V_1$. 
This implies that the edges between $V_1$ and $V_2$ are all colored the same in $G$, and that $G[V_2]\in\mathcal{G}_{j}$.
The choice of $G$ implies that $G[V_1]\in\mathcal{G}_{i+1}$, and thus $i+1+j=k$.
It follows that $G\in\mathcal{G}_k$, which completes the proof.
\end{proof}

The proof of Theorem~\ref{thm:sum} relies on a relation between directed triangles in digraphs and rainbow triangles in edge-colored graphs, similar to an approach in~\cite{linixuzh}.
However, a more careful analysis is needed as we seek multiple rainbow triangles. 
In order to be self-contained, we include all the details.

For a digraph $D$, we denote by $a(D)$ the number of arcs, by $N^+_D(v)$ and  $N^-_D(v)$ the out- and in-neighborhood, and by $d^+_D(v)$ and $d^-_D(v)$ the out- and in-degree, of a vertex $v$ in $V(D)$, respectively. 
Let the {\it out-component number} $\omega^+_D(v)$ denote the number of weak components of $D[N_D^+(v)]$ for a vertex $v$ in $D$.
We can assign colors to the arcs of $D$ in the following way. 
For $v\in V(D)$ and a weak component $H$ of $D[N^+(v)]$, we color all arcs from $v$ to $H$ the same with a new color. 
That is, two arcs $(w,x)$ and $(y,z)$ have the same color if and only if $w=y$, and $x$ and $z$ are in the same weak component of $D[N^+(y)]$. 
We call the underlying undirected graph with the same edge coloring the {\it associated colored graph of $D$}. 

We suitably extend a result of~\cite{linixuzh} for our purposes.
\setcounter{claim}{0}
\begin{lemma}\label{thm:directed}
If $D$ is an oriented graph on $n$ vertices such that $a(D)+\sum\limits_{v\in V(D)} \omega_D^+(v)\geq \binom{n+1}{2}+k-1$ for a positive integer $k$, then $D$ contains $k$ directed triangles.
\end{lemma}

\begin{proof}
Let $D$ be as in the statement of the theorem and let $G$ be the associated colored graph of $D$. 
We show that for any three distinct vertices $u,v,$ and $w$ of $D$, $uvw$ is a directed triangle in $D$ if and only if $uvw$ is a rainbow triangle in $G$. 
If $uvw$ is a directed triangle in $D$, by the construction of the associated colored graph, we know that all the edges of the underlying triangle are colored pairwise differently.
Hence, $uvw$ is a rainbow triangle in $G$. 
Now, assume that $uvw$ is a rainbow triangle in $G$, and $uvw$ is not a directed triangle in $D$. 
Hence, we may assume that one of the vertices dominates the others, say $u$.
By the construction of $G$, the edges $uv$ and $uw$ are colored with the same color, a contradiction. 
	
Since $m(G)=a(D)$ and $c(G)=\sum\limits_{v\in V(D)} \omega_D^+(v)$, we obtain that
$$m(G)+c(G)=a(D)+\sum\limits_{v\in V(D)} \omega_D^+(v)\geq \binom{n+1}{2}+k-1.$$
Hence, Theorem \ref{thm:m+c} implies that $G$ contains at least $k$ rainbow triangles. Since every rainbow triangle in $G$ corresponds to a directed triangle in $D$, we obtain that $D$ contains at least $k$ directed triangles. 
\end{proof}

We proceed to the proof of Theorem~\ref{thm:sum}, which is based on Lemma~\ref{thm:directed}.
\begin{proof}[Proof of Theorem~\ref{thm:sum}]
Suppose, for a contradiction, that the statement is false and let $G$ be a counterexample chosen such that 
\begin{enumerate}[label=(\roman*)]
\item $k$ is as small as possible,
\item subject to (i), the number of edges $m(G)$ is as small as possible.
\end{enumerate}

The case for $k=1$ was proved in~\cite{linixuzh}. Hence, we may assume that $k\geq2$, and, thus, by the choice of $G$, we can find at least $k-1$ rainbow triangles in $G$.

\begin{claim}\label{claim3}
All rainbow triangles in $G$ are edge-disjoint.
\end{claim}
\begin{proof}
Suppose there is an edge $e$ that is contained in at least two rainbow triangles in $G$.
Then, $\sum\limits_{v\in V(G-e)} d^c_{G-e}(v)\geq \binom{n+1}{2}+k-3$, which implies that $G-e$ contains at least $k-2$ rainbow triangles by the choice of $G$.  
Hence, $G$ contains at least $k$ rainbow triangles, a contradiction.
\end{proof}

\begin{claim}\label{claim4}
G does not contain a monochromatic $P_4$.
\end{claim}
\begin{proof}
Suppose, for a contradiction, that $abcd$ is a monochromatic path on 4 vertices. 
Then, $\sum\limits_{v\in V(G)} d^c_G(v)
=\sum\limits_{v\in V(G-bc)} d^c_{G-bc}(v)$, which contradicts the choice of $G$.
\end{proof}
\begin{claim}\label{claim5}
No edge of a rainbow triangle is contained in a monochromatic $P_3$.
\end{claim}
\begin{proof}
Suppose, for a contradiction, that $abc$ is a rainbow triangle of $G$ and that $ax$ is an edge not contained in $abc$ such that $ax$ and $ab$ are colored the same.
Since, 
$$\sum\limits_{v\in V(G-ab)} d^c_{G-ab}(v)
=\left(\sum\limits_{v\in V(G)} d^c_G(v)\right)-1\geq \binom{n+1}{2}+k-2,$$ 
$G-ab$ contains at least $k-1$ rainbow triangles by the choice of $G$.
Hence, $G$ contains at least $k$ rainbow triangles, a contradiction.
\end{proof}

We assign an orientation to $G$ as follows. 
For any monochromatic $P_3$, say $abc$, we direct the two edges away from $b$.
By Claim~\ref{claim4}, $G$ does not contain a monochromatic $P_4$ and so the remaining edges can be directed arbitrarily. 
Thus, by Claim~\ref{claim5}, the edges of a rainbow triangle can be directed arbitrarily and by Claim~\ref{claim3}, we may assume that all rainbow triangles in $G$ correspond to directed triangles. 
Let $D$ denote the resulting oriented graph.

Note that this orientation also implies that 
for an arc $(u,v)$ all edges incident to $v$ have a different color than $uv$.
This immediately implies the following.
\begin{claim}\label{claim6}
If $u$ is a vertex of $D$ and $v,w$ are adjacent out-neighbors of $u$ such that $uvw$ is not a rainbow triangle in $G$, then $uv$ and $uw$ have the same color. 
\end{claim}
\begin{claim}\label{claim7}
For every vertex $u$ of $D$, $d_D^-(u)+\omega_D^+(u)\geq d^c_G(u)$.
\end{claim}
\begin{proof}
First, assume that $u$ is a vertex that is not contained in a rainbow triangle of $G$.
By iteratively applying Claim~\ref{claim6}, all arcs from $u$ into one component of $D[N^+(u)]$ have the same color. 
Hence, $d_D^-(u)+\omega_D^+(u)\geq d^c_G(u)$.

Now, assume that $u$ is a vertex contained in a rainbow triangle $uvw$ of $G$ that corresponds to a directed triangle $uvw$ in $D$.
Let $D_v$ be the component of $D[N^+(u)]$ that contains $v$. 
We show that $D_v=\{v\}$.
Suppose, for a contradiction, that there is another out-neighbor $x$ of $u$ that is adjacent to $v$. 
Note that by Claim~\ref{claim5}, the edges $ux$ and $vx$ are colored differently than $uv$ in G. Since $ux$ is directed towards $x$ in $D$, the edge $vx$ is also colored differently than $ux$ in $G$. 
This implies that $uvx$ is a rainbow triangle in $G$, which contradicts Claim~\ref{claim3}.
Hence, $D_v=\{v\}$, and, therefore, similar as above, we obtain that $d_D^-(u)+\omega_D^+(u)\geq d^c_G(u)$. 
\end{proof}
By Claim~\ref{claim7}, we obtain that 
$$a(D)+\sum\limits_{v\in V(D)}\omega_D^+(v)=\sum\limits_{v\in V(D)} \left( d^-_D(v)+\omega_D^+(v)\right)\geq \sum\limits_{v\in V(G)}d^c_G(u)\geq \binom{n+1}{2}+k-1.$$
Hence, by Theorem~\ref{thm:directed}, $D$ contains at least $k$ directed triangles. By the choice of the orientation, every directed triangle in $D$ corresponds to a rainbow triangle in $G$, and, thus, $G$ contains at least $k$ rainbow triangles, which yields the final contradiction.
\end{proof}

\section{Rainbow cliques in edge-colored graphs}\label{sec:3}
The proof of Theorem~\ref{thm:Kkchar} is based on the following three lemmas. 
We first introduce some basic notation for Tur\'an graphs.

The Tur\'an graph $T_{n,k}$ is the complete $k$-partite graph on $n$ vertices with balanced partite sets. 
More precisely, for $k\leq n$ and $n=p k+i$ with $p=\lfloor\frac{n}{k}\rfloor$ and $0\leq i\leq k-1$, $T_{n,k}$ has partite sets $V_1,V_2,\ldots,V_k$, 
where $|V_j|=p+1$ for $1\leq j\leq i$ and $|V_j|=p$ for $i+1\leq j\leq k$. 
It is easy to see that the number of edges in $T_{n,k}$ is
\begin{eqnarray}\label{eq:tnk}
t_{n,k}
=\binom{k}{2}p^2+i(k-1)p+\binom{i}{2}
=\frac{k-1}{2k}\left(n^2-i^2\right)+\binom{i}{2}.
\end{eqnarray} 
We will also use the following simple fact that
\begin{eqnarray}\label{eq:difftnk}
t_{n+1,k}-t_{n,k}
=n-\frac{n-i}{k}.
\end{eqnarray}

\begin{lemma}\label{lem:kgeq5}
Let $G$ be an edge-colored complete graph on $n$ vertices that contains no rainbow $K_k$ but a rainbow $T_{n,k-2}$ as a subgraph with $n\geq k\geq 6$ and $\lfloor\frac{n}{k-2}\rfloor\geq 2$. 
If the edges of $G$ are colored with $t_{n,k-2}+1$ colors then all edges within the partite sets of the rainbow $T_{n,k-2}$ are colored with the same color and this color is not used on the rainbow $T_{n,k-2}$.
\end{lemma}

\begin{proof}
Let $G$ be a graph as described above. 
We may assume that $1,2,\dots,t_{n,k-2}$ are the colors used to color the edges of the rainbow $T_{n,k-2}$. 
Since the edges of the graph are colored with $t_{n,k-2}+1$ colors, at least one edge within a partite set is colored with a different color, say color $0$. 

Suppose, for a contradiction, that the statement is false. 
In this case there is an edge within a partite set that is not colored with color $0$. 
Without loss of generality we may assume that the edges $v_1v_2$, $v_3v_4$, $v_5v_6$, and $v_7v_8$ lie in different partite sets and that $v_1v_2$ is colored with color $0$ and $v_3v_4$ is colored with color $1$. 
If no edge between the pairs of vertices $v_1,v_2$ and $v_3,v_4$ is colored with color $1$, the graph induced on the vertices $v_1,v_2,v_3,v_4$ and one suitable vertex from every other partite set form a rainbow $K_k$. 
This suitable choice is possible, because all partite sets contain at least $2$ vertices. 
Hence, by symmetry, we may assume that $v_1v_3$ is colored with color $1$. 
If $v_5v_6$ is colored with color $0$, the vertices $v_2,v_3,v_4,v_5,v_6$ together with a vertex of every other partite set induce a rainbow $K_k$. 
If $v_5v_6$ is colored with color $1$, the vertices $v_1,v_2,v_4,v_5,v_6$ together with a vertex of every other partite set induce a rainbow $K_k$. Therefore, we may assume that the edge $v_5,v_6$ is colored with color $2$. 

With the same argument as above, we know that one edge between the pairs of vertices $v_1,v_2$ and $v_5,v_6$ is colored with color $2$. 
If this edge is incident with $v_1$, the graph induced by $v_2,v_3,v_4,v_5,v_6$ together with one vertex of every other partite set is a rainbow $K_k$. 
Hence, by symmetry, we may assume that $v_2v_5$ is colored with color $2$. 
With the same argument as above, we know that the edge $v_7v_8$ is colored with a color  different from $0$, $1$, and $2$. 
We may assume that this edge is colored with color $3$. 
If no edge between the pairs of vertices $v_1,v_2$ and $v_7,v_8$ is colored with color $3$, the graph induced on the vertices $v_1,v_2,v_7,v_8$ and one suitable vertex of every other partite set form a rainbow $K_k$. 
By symmetry, we may assume that the edge $v_1v_7$ is colored with color $3$. 
But then the vertices $v_2,v_3,v_4,v_6,v_7,v_8$ and one vertex out of every other partite set form a rainbow $K_k$, a final contradiction. 
\end{proof}

\begin{lemma}\label{lem:rbTuran}
If $G$ is a complete graph on $n$ vertices that is edge-colored with $t_{n,k-2}+1$ colors and does not contain a rainbow $K_k$ for $n\geq k\geq6$, then $G$ contains a rainbow $T_{n,k-2}$ as a subgraph.
\end{lemma}
\begin{proof}
We prove the lemma by induction on $n$ for every $k\geq6$.
For $n=k$, we obtain that $G$ is edge-colored with $t_{n,n-2}+1=\binom{n}{2}-1$ colors and we can easily find a rainbow $T_{n,n-2}$ as a subgraph.

So suppose that $G$ is a complete graph on $n+1> k$ vertices that is edge-colored with $t_{n+1,k-2}+1$ colors and does not contain a rainbow $K_k$.
Let $q=k-2$ and $n= pq+i$ for $0\leq i\leq q-1$. 
We consider the following two cases.
\begin{case}\label{case1}
There is a vertex $v$ of $G$ such that $c(G-v)\geq t_{n,q}+1$.
\end{case}
Let $v$ be as in Case~\ref{case1}. 
If $c(G-v)\geq t_{n,q}+2$, then $G-v$ contains a rainbow $K_k$ by Theorem~\ref{thm:xucliques}, a contradiction.
Thus, we may assume that $c(G-v)= t_{n,q}+1$. By induction, we can find a rainbow $T_{n,q}$ as a subgraph in $G-v$.
Together with $c(G-v)=c(G)-d_G^s(v)$ and (\ref{eq:difftnk}), we obtain that 
$$d_G^s(v)=t_{n+1,q}-t_{n,q}=
n-\frac{n-i}{q}=n-\left\lfloor\frac{n}{q}\right\rfloor.$$

If $\lfloor\frac{n}{q}\rfloor>1$, then we know that by Lemma \ref{lem:kgeq5} all edges within the partition classes of the $T_{n,q}$ subgraph in $G-v$ are colored with the same color, a color that is not used to color the edges of the rainbow $T_{n,q}$. Hence, we can suitably add $v$ to the rainbow subgraph $T_{n,q}$ and obtain a rainbow $T_{n+1,q}$ in $G$, since $G$ has no rainbow $K_k$ as a subgraph.

If $\lfloor\frac{n}{q}\rfloor=1$, then $d_G^s(v)=n-1$. 
Hence, there are $n-1$ many edges incident with $v$ that are $v$-saturated and have pairwise different colors. 
Let $uv$ denote one edge such that the set of all other edges incident with $v$ fulfills the above. If $u$ is in a partite set of size one of the rainbow $T_{n,q}$, we can add $v$ to that partite set and we obtain a rainbow $T_{n+1,q}$ in $G$. Thus, we may assume that the vertex $u$ is in a partite set of size two. Note that there is an edge in $G-v$ that is colored with a color not used to color the rainbow $T_{n,q}$ in $G-v$. Let $wx$ denote this edge. 
If $u\not=w$ and $u\not=x$, the vertices $v$, $w$, and $x$ together with a suitable vertex of every other partite set induce a rainbow $K_k$ in $G$. Hence, by symmetry we may assume that $u=x$. Now one can see that $G$ contains a rainbow $T_{n+1,q}$ as a subgraph:  The vertices $v$ and $u$ form one partite set, and $w$ is added to a partite set of size one of the rainbow $T_{n,q}$ in $G-v$. The other partite sets are the same as in the rainbow $T_{n,q}$ in $G-v$. 

\begin{case}\label{case2}
$c(G-v)\leq t_{n,q}$ for every vertex $v$ of $G$.
\end{case}
With this assumption we obtain that
\begin{align*}
(n+1)c(G) &= \sum\limits_{v\in V(G)}\left(c(G-v)+d_G^s(v)\right) 
\\&\leq (n+1)t_{n,q}+\sum\limits_{v\in V(G)}d_G^s(v)
\\&\leq (n+1)t_{n,q}+2c(G),
\end{align*}
which implies 
\begin{alignat*}{2}
&&(n-1)(t_{n+1,q}+1)&\leq (n+1)t_{n,q}\\
&\Leftrightarrow\quad &(n-1)(t_{n+1,q}-t_{n,q}+1)&\leq 2t_{n,q}.
\end{alignat*}

Together with (\ref{eq:tnk}) and (\ref{eq:difftnk}) we obtain that
\begin{alignat*}{1}
(n-1)\left(n-\frac{n-i}{q}+1\right)
&\leq 2\left(\frac{q-1}{2q} \left(n^2-i^2\right)+\binom{i}{2}\right)
\\&= \left(1-\frac{1}{q}\right)\left(n^2-i^2\right)+i(i-1)
\end{alignat*}
and thus
\begin{alignat}{2}
&&n^2-\frac{n^2-in}{q}+n-n+\frac{n-i}{q}-1 &\leq n^2-\frac{n^2-i^2}{q}-i
\notag\\ &\Leftrightarrow &\frac{n-i+in}{q} &\leq 1+\frac{i^2}{q}-i
\notag\\ &\Leftrightarrow &n-i+in &\leq q+i^2-iq
\notag\\ &\Leftrightarrow &n &\leq \frac{i+i^2+q(1-i)}{1+i} =i+\frac{q(1-i)}{1+i} \label{eq:n}.
\end{alignat}
For $0\leq i\leq q-1$, (\ref{eq:n}) yields that $n\leq q$. 
This implies the final contradiction for Case~\ref{case2} since $n\geq k=q+2$.
\end{proof}

\begin{lemma}\label{lem:Tnkcomp}
If $G$ is an edge-colored graph on $n$ vertices such that $m(G)+c(G)=\binom{n}{2}+t_{n,k-2}+1$, and $G$ does not contain a rainbow $K_k$ for $n\geq k\geq6$, then $G$ is complete.
\end{lemma}

\begin{proof}
We prove the lemma by induction on $n$ for every $k\geq6$.
For $n=k$, we obtain that $m(G)+c(G)=\binom{n}{2}+t_{n,n-2}+1=\binom{n}{2}+\binom{n}{2}-1$, which implies that $G$ is complete.

So suppose that $G$ is a graph on $n+1>k$ vertices such that  $m(G)+c(G)=\binom{n+1}{2}+t_{n+1,k-2}+1$ and $G$ does not contain a rainbow $K_k$.
Let $q=k-2$ and $n= pq+i$ for $0\leq i\leq q-1$. 
As in the proof of Lemma~\ref{lem:rbTuran}, we consider the following two cases.
\setcounter{case}{0}

\begin{case}\label{case3}
There is a vertex $v$ of $G$ such that $\displaystyle m(G-v)+c(G-v)\geq \binom{n}{2}+ t_{n,q}+1$.
\end{case}
We have 
$$d_G(v)+d_G^s(v)\geq \binom{n+1}{2}-\binom{n}{2}+t_{n+1,q}-t_{n,q}
=2n-\frac{n-i}{q}
=2n-\left\lfloor\frac{n}{q}\right\rfloor,$$
because otherwise, it holds that
\begin{align*}
m(G-v)+c(G-v)&=m(G)+c(G)-(d_G(v)+d_G^s(v))\\
&\geq \binom{n+1}{2}+t_{n+1,q}+1 -\left(\binom{n+1}{2}-\binom{n}{2}+t_{n+1,q}-t_{n,q} -1\right)\\
& = \binom{n}{2}+t_{n,q}+2
\end{align*}
and by Theorem \ref{thm:xucliques}, the graph $G-v$ contains a rainbow $K_k$, a contradiction.

With the same argument we obtain that $m(G-v)+c(G-v)= \binom{n}{2}+ t_{n,q}+1$ and by induction, it holds that $G-v$ is complete. 
Hence, $G-v$ is a complete graph that is edge-colored with $t_{n,q}+1$ colors, and by Lemma \ref{lem:rbTuran}, we have that the graph $G-v$ contains a rainbow $T_{n,q}$ as a subgraph.

If $\lfloor\frac{n}{q}\rfloor=1$, we obtain $d_G(v)+d_G^s(v)\geq 2n-1$, and thus, the vertex $v$ must have degree $n$ and $G$ is complete.

So let $\lfloor\frac{n}{q}\rfloor>1$. By Lemma \ref{lem:kgeq5}, all edges of the partite sets of the rainbow $T_{n,q}$ are colored with the same color that is not used to color the rainbow $T_{n,q}$. 

Suppose for a contradiction that $G$ is not complete, that is, $d_G(v)\leq n-1$.
We obtain that
\begin{eqnarray}\label{eq:ds}
d_G^s(v)\geq 2n-\left\lfloor\frac{n}{q}\right\rfloor-d(v)\geq n+1-\left\lfloor\frac{n}{q}\right\rfloor.
\end{eqnarray}

The partitions of $T_{n,q}$ are either of size $\lceil \frac{n}{q}\rceil$ or $\lfloor\frac{n}{q}\rfloor$. 
Hence, the vertex $v$ has at least one $v$-saturated edge to every partite set and to one partite set at least two $v$-saturated edges. 
But in this case the graph induced by $v$ and those $k-1$ neighbors is a rainbow $K_k$, which yields the final contradiction for Case~\ref{case3}.   

\begin{case}\label{case4}
$\displaystyle m(G-v)+c(G-v)\leq \binom{n}{2}+t_{n,q}$ for every vertex $v$ of $G$.
\end{case}
With this assumption we obtain that
\begin{align*}
(n+1)(m(G)+c(G)) &= \sum\limits_{v\in V(G)}\Big(m(G-v)+c(G-v)+d(v)+d^s(v)\Big) 
\\&\leq (n+1)\left(\binom{n}{2}+t_{n,q}\right)+2m(G)+2c(G),
\end{align*}
which implies 
\begin{alignat*}{2}
&&(n-1)\left(\binom{n+1}{2}+t_{n+1,q}+1\right)&\leq (n+1)\left(\binom{n}{2}+t_{n,q}\right)\\
&\Leftrightarrow\quad &(n-1)(t_{n+1,q}-t_{n,q})+(n-1)n+n-1&\leq 2\left(t_{n,q}+\binom{n}{2}\right).
\end{alignat*}

Together with (\ref{eq:tnk}) and (\ref{eq:difftnk}) we have
\begin{alignat}{2}
&&(n-1)\left(n-\frac{n-i}{q}\right)+(n-1)(n+1)
&\leq 2\left(\frac{q-1}{2q} \left(n^2-i^2\right)+\binom{i}{2}\right)+n(n-1)
\notag\\&\Leftrightarrow\quad &n^2-\frac{n^2-in}{q}-n+\frac{n-i}{q}+n-1
&\leq \left(1-\frac{1}{q}\right)\left(n^2-i^2\right)+i(i-1)
\notag\\&\Leftrightarrow\quad 
&\frac{n-i+in}{q}-n+n-1
&\leq \frac{i^2}{q}-i
\notag\\&\Leftrightarrow\quad 
&\frac{n-i+in}{q}
&\leq 1+\frac{i^2}{q}-i
\notag\\&\Leftrightarrow\quad 
&n&\leq i+\frac{q(1-i)}{1+i}\label{eq:n2}.
\end{alignat}
Analogously as for equation~(\ref{eq:n}) in Lemma~\ref{lem:rbTuran}, we obtain the final contradiction for Case~\ref{case4} since for $0\leq i\leq q-1$, (\ref{eq:n2}) yields that $n\leq q$ but $n\geq k=q+2$. 
\end{proof}

\begin{proof}[Proof of Theorem~\ref{thm:Kkchar}]
This follows with Lemmas~\ref{lem:kgeq5}, \ref{lem:rbTuran},  and~\ref{lem:Tnkcomp}.
\end{proof}

\section{Concluding remarks}
As main results in this paper we gave sufficient conditions for multiple rainbow triangles in edge-colored graphs and characterized corresponding extremal graphs.
We also characterized the extremal graphs, where $m(G)+c(G)$ is maximum but that do not contain a rainbow clique of size~$k$.
In particular, this answers two questions raised in~\cite{funixuzh}.

Note that Theorem~\ref{thm:Kkchar} only characterizes the extremal graphs without a rainbow $K_k$ in Theorem~\ref{thm:xucliques} for $k\geq6$. 
For $k\in\{4,5\}$, one has to distinguish many different cases, which seems to be tedious and gives only little hope for a nice characterization. 
Figure~\ref{fig3} shows several examples of extremal graphs for $k\in\{4,5\}$, where the coloring within the partite sets can be very particular and/or the graphs are not complete.

\begin{figure}[htb]
\begin{center}
\begin{tikzpicture}

{\begin{scope}[]
\tikzstyle{enddot}=[circle,inner sep=0cm, minimum size=10pt,color=hellgrau,fill=hellgrau]
\tikzstyle{markline}=[draw=hellgrau,line width=10pt]

\def\csizes{{2,2}} 
\def\cnum{1}         
\def\radius{2}

\def\startangle{90}

\pgfmathsetmacro{\cnumminusone}{\cnum-1}
\pgfmathsetmacro{\classrange}{360/(\cnum+1)}

\foreach \i in {0,...,\cnum}{
  \pgfmathparse{\csizes[\i]}
  \pgfmathsetmacro{\gap}{\classrange/(\pgfmathresult+1)}
  
  \pgfmathsetmacro{\markangle}{0.8*\classrange}  

  \pgfmathsetmacro{\sangle}{\startangle+\i*\classrange+\classrange+0.1*\classrange}  
  \draw[markline] (\sangle:\radius) arc [radius=\radius, start angle=\sangle, delta angle=\markangle];

  \foreach \vx in {1,...,\pgfmathresult}{
    \node[hvertex] (a\i b\vx) at (\startangle+\i*\classrange+\classrange+\vx*\gap:\radius){};
  }
}

\foreach \i in {0,...,\cnumminusone}{
  \pgfmathtruncatemacro{\plusone}{\i+1}
  \foreach \j in {\plusone,...,\cnum}{
     \pgfmathparse{\csizes[\i]}
     \foreach \u in {1,...,\pgfmathresult}{
       \pgfmathparse{\csizes[\j]}
        \foreach \v in {1,...,\pgfmathresult}{
           \draw[hedge] (a\i b\u) -- (a\j b\v);
        }
     }
  }
}

\draw[ultra thick,bend right=22,blue] (a0b1) to (a0b2);
\draw[ultra thick,blue] (a0b2) to (a1b1);

\draw[ultra thick,bend right=22,red] (a1b1) to (a1b2);

\end{scope}}

{\begin{scope}[shift={(5cm,0cm)}]
\tikzstyle{enddot}=[circle,inner sep=0cm, minimum size=10pt,color=hellgrau,fill=hellgrau]
\tikzstyle{markline}=[draw=hellgrau,line width=10pt]

\def\csizes{{2,3}} 
\def\cnum{1}         
\def\radius{2}

\def\startangle{90}

\pgfmathsetmacro{\cnumminusone}{\cnum-1}
\pgfmathsetmacro{\classrange}{360/(\cnum+1)}


\foreach \i in {0,...,\cnum}{
  \pgfmathparse{\csizes[\i]}
  \pgfmathsetmacro{\gap}{\classrange/(\pgfmathresult+1)}
  
  \pgfmathsetmacro{\markangle}{0.8*\classrange}  

  \pgfmathsetmacro{\sangle}{\startangle+\i*\classrange+\classrange+0.1*\classrange}  
  \draw[markline] (\sangle:\radius) arc [radius=\radius, start angle=\sangle, delta angle=\markangle];

  \foreach \vx in {1,...,\pgfmathresult}{
    \node[hvertex] (a\i b\vx) at (\startangle+\i*\classrange+\classrange+\vx*\gap:\radius){};
  }
}

\foreach \i in {0,...,\cnumminusone}{
  \pgfmathtruncatemacro{\plusone}{\i+1}
  \foreach \j in {\plusone,...,\cnum}{
     \pgfmathparse{\csizes[\i]}
     \foreach \u in {1,...,\pgfmathresult}{
       \pgfmathparse{\csizes[\j]}
        \foreach \v in {1,...,\pgfmathresult}{
           \draw[hedge] (a\i b\u) -- (a\j b\v);
        }
     }
  }
}

\draw[ultra thick,bend right=18,color=red!45!black] (a1b1) to (a1b2);
\draw[ultra thick,bend right=18,red] (a1b2) to (a1b3);

\draw[ultra thick,bend right=22,blue] (a0b1) to (a0b2);
\draw[ultra thick,blue] (a0b2) to (a1b2);

\end{scope}}

{\begin{scope}[shift={(0cm,-5cm)}]
\tikzstyle{enddot}=[circle,inner sep=0cm, minimum size=10pt,color=hellgrau,fill=hellgrau]
\tikzstyle{markline}=[draw=hellgrau,line width=10pt]

\def\csizes{{2,2,2}} 
\def\cnum{2}         
\def\radius{2}

\def\startangle{90}

\pgfmathsetmacro{\cnumminusone}{\cnum-1}
\pgfmathsetmacro{\classrange}{360/(\cnum+1)}


\foreach \i in {0,...,\cnum}{
  \pgfmathparse{\csizes[\i]}
  \pgfmathsetmacro{\gap}{\classrange/(\pgfmathresult+1)}
  
  \pgfmathsetmacro{\markangle}{0.8*\classrange}  

  \pgfmathsetmacro{\sangle}{\startangle+\i*\classrange+\classrange+0.1*\classrange}  
  \draw[markline] (\sangle:\radius) arc [radius=\radius, start angle=\sangle, delta angle=\markangle];

  \foreach \vx in {1,...,\pgfmathresult}{
    \node[hvertex] (a\i b\vx) at (\startangle+\i*\classrange+\classrange+\vx*\gap:\radius){};
  }
}

\foreach \i in {0,...,\cnumminusone}{
  \pgfmathtruncatemacro{\plusone}{\i+1}
  \foreach \j in {\plusone,...,\cnum}{
     \pgfmathparse{\csizes[\i]}
     \foreach \u in {1,...,\pgfmathresult}{
       \pgfmathparse{\csizes[\j]}
        \foreach \v in {1,...,\pgfmathresult}{
           \draw[hedge] (a\i b\u) -- (a\j b\v);
        }
     }
  }
}

\draw[ultra thick,bend right=15,blue] (a2b1) to (a2b2);
\draw[ultra thick,blue] (a2b2) -- (a0b1);

\draw[ultra thick,bend right=15,color=green!50!black] (a1b1) to (a1b2);
\draw[ultra thick,color=green!50!black] (a1b1) -- (a0b2);

\draw[ultra thick,bend right=15,red] (a0b1) to (a0b2);

\end{scope}}

{\begin{scope}[shift={(5cm,-5cm)}]
\tikzstyle{enddot}=[circle,inner sep=0cm, minimum size=10pt,color=hellgrau,fill=hellgrau]
\tikzstyle{markline}=[draw=hellgrau,line width=10pt]

\def\csizes{{2,3,3}} 
\def\cnum{2}         
\def\radius{2}

\def\startangle{90}

\pgfmathsetmacro{\cnumminusone}{\cnum-1}
\pgfmathsetmacro{\classrange}{360/(\cnum+1)}

\foreach \i in {0,...,\cnum}{
  \pgfmathparse{\csizes[\i]}
  \pgfmathsetmacro{\gap}{\classrange/(\pgfmathresult+1)}
  
  \pgfmathsetmacro{\markangle}{0.8*\classrange}  

  \pgfmathsetmacro{\sangle}{\startangle+\i*\classrange+\classrange+0.1*\classrange}  
  \draw[markline] (\sangle:\radius) arc [radius=\radius, start angle=\sangle, delta angle=\markangle];

  \foreach \vx in {1,...,\pgfmathresult}{
    \node[hvertex] (a\i b\vx) at (\startangle+\i*\classrange+\classrange+\vx*\gap:\radius){};
  }
}

\foreach \i in {0,...,\cnumminusone}{
  \pgfmathtruncatemacro{\plusone}{\i+1}
  \foreach \j in {\plusone,...,\cnum}{
     \pgfmathparse{\csizes[\i]}
     \foreach \u in {1,...,\pgfmathresult}{
       \pgfmathparse{\csizes[\j]}
        \foreach \v in {1,...,\pgfmathresult}{
           \draw[hedge] (a\i b\u) -- (a\j b\v);
        }
     }
  }
}

\draw[ultra thick,blue] (a2b1) -- (a2b2);
\draw[ultra thick,blue] (a2b2) -- (a0b1);
\draw[ultra thick,blue] (a2b2) -- (a2b3);

\draw[ultra thick,color=green!50!black] (a1b1) -- (a1b2);
\draw[ultra thick,color=green!50!black] (a1b2) -- (a1b3);
\draw[ultra thick,color=green!50!black] (a1b2) -- (a0b2);

\draw[ultra thick,bend right=15,red] (a0b1) to (a0b2);

\draw[ultra thick,bend right=45,color=green] (a1b1) to (a1b3);
\draw[ultra thick,green] (a1b1) -- (a0b2);

\draw[ultra thick,bend right=45,cyan] (a2b1) to (a2b3);
\draw[ultra thick,cyan] (a2b3) -- (a0b1);

\end{scope}}

\end{tikzpicture}
\end{center}
\caption{\onehalfspacing 
Extremal graphs for $k=4$ (top) and $k=5$ (bottom) that do not contain a rainbow $K_k$ such that $m(G)+c(G)=\binom{n}{2}+t_{n,k-2}+1$. 
The thin lines represent the edges of a rainbow $T_{n,k-2}$ and use different colors than the colors shown in the figure.
Note that the graph on the top right is not complete.
}\label{fig3}
\end{figure}
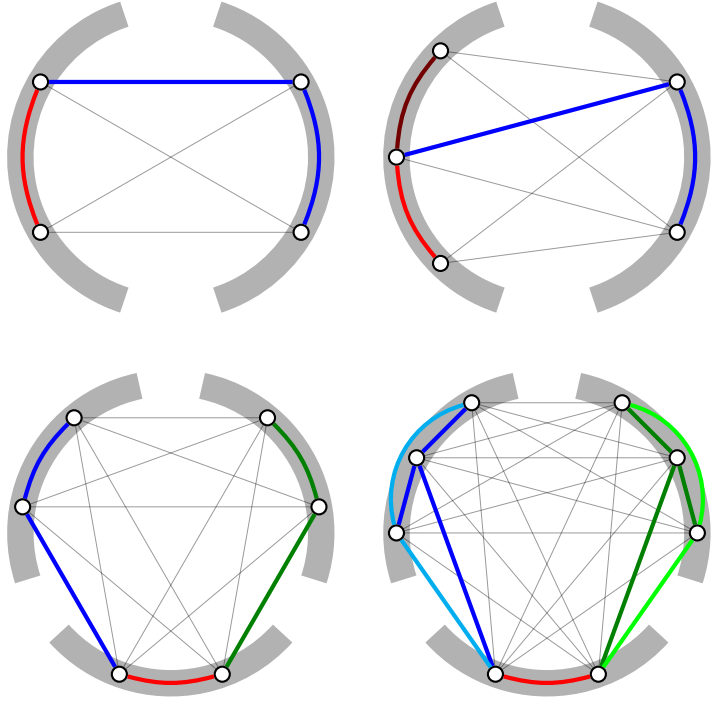

In 1962, Erd\H{o}s~\cite{er2,er3} generalized Rademacher's theorem for the existence of many triangles to arbitrarily sized cliques.
This remains open for arbitrarily sized rainbow cliques. 
That is, whether one can show that the supersaturated setting in Theorem~\ref{thm:xucliques} ensures not only one but many rainbow $K_k$'s for $k\geq4$.

At least, we can easily guarantee the existence of many rainbow $K_k$'s in the following sense.

\begin{proposition}
If $G$ is an edge-colored graph on $n$ vertices and $n\geq k\geq 4$ such that $m(G)+c(G)\geq \binom{n}{2}+t_{n,k-2}+2\ell$ for a positive integer $\ell$, then $G$ contains at least $\ell$ rainbow $K_k$'s.
\end{proposition}

\begin{proof}
We prove the statement by induction on $\ell$. 
For $\ell=1$, the statement is true by Theorem~\ref{thm:xucliques}.
Hence, we may assume that $\ell\geq2$ and $G$ contains at least $\ell-1$ rainbow $K_k$'s. 
Let $e$ be an edge that is contained in a rainbow $K_k$.
Then, $m(G-e)+c(G-e)\geq \binom{n}{2}+t_{n,k-2}+2\ell-2$, and thus by induction, $G-e$ contains at least $\ell-1$ rainbow $K_k$'s.
This implies that $G$ contains at least $\ell$ rainbow $K_k$'s.
\end{proof}

\pagebreak

\end{document}